\documentclass{article}
\usepackage{amssymb, amsmath, amsthm, mathabx, authblk}
\usepackage{helvet}        
\usepackage{courier}       
\usepackage{type1cm}       
\usepackage{makeidx}       
\usepackage{graphicx}      
\usepackage{multicol}       
\usepackage[bottom]{footmisc}
\usepackage[all]{xy}

\newtheorem{theorem}{Theorem}[section] 
\newtheorem{lemma}[theorem]{Lemma}     

\newtheorem{proposition}[theorem]{Proposition}

\newtheorem{conjecture}{Conjecture}  
\newtheorem{definition}{Definition}
\newtheorem{remark}{Remark}
\newtheorem{example}{Example}

\title{Kato's local epsilon conjecture: $l \neq p$ case} 
\author{Mahesh Kakde \\
mahesh.kakde@kcl.ac.uk \\
King's College London}

\newcommand{\ilim}[1]{%
	\displaystyle{%
	\lim_{\genfrac{}{}{0pt}{}{\longleftarrow}{\scriptstyle #1}} }\;}

\begin{document}
\maketitle
\tableofcontents

\begin{abstract} Let $l$ and $p$ be two distinct primes. Let $K$ be a local field of characteristic 0 and residue characteristic $l$. In this paper we prove existence of local $\epsilon_0$-constants for representations of $Gal(\overline{K}/K)$ over Iwasawa algebras of $p$-adic Lie groups. Existence of these $\epsilon_0$-constants was conjectured by Kato (for commutative Iwasawa algebras) and Fukaya-Kato (in general).
\end{abstract}

\section{Introduction} In the last decade a nice conjectural picture of noncommutative Iwasawa theory has emerged thanks to Burns-Flach \cite{BurnsFlach:2001}, Huber-Kings \cite{HuberKings:2002}, Coates. \emph{et. al.} \cite{CFKSV:2005}, Fukaya-Kato \cite{FukayaKato:2006}, among others. A key part of Iwasawa theory is the (conjectural) existence of $p$-adic $L$-function for a motive and a $p$-adic Lie extension of a number field, both satisfying certain conditions (see section 4 of Fukaya-Kato \cite{FukayaKato:2006}). The defining property of the $p$-adic $L$-function is that it interpolates certain values of complex $L$-functions coming from the motive and Artin representations of the $p$-adic Lie extension. $p$-adic $L$-functions are expected to satisfy a functional equation analogous to (conjectural) functional equations of complex $L$-functions. The so called epsilon constant appears in the functional equation of complex $L$-functions. This epsilon constant is a product of local epsilon constants; $l$-adic epsilon constant for each prime $l$ of the number field over which we are defining the $L$-function. These local epsilon constants are defined for one dimensional representations in Tate's thesis \cite{Tate:1950}. They are constructed in general in an unpublished manuscript of Langlands \cite{Langlands:1970} and later, by a simpler method, in Deligne \cite{Deligne:1973}. Kato \cite{Kato:1993b} and Fukaya-Kato \cite{FukayaKato:2006}  conjecture existence of $p$-adic analogues of these epsilon constants. We call this conjecture `Kato's local $\epsilon$-conjecture'. These $p$-adic analogues of epsilon constants should appear in the conjectural functional equation of $p$-adic $L$-functions. Their defining property is that they interpolate the epsilon constants of Deligne-Langlands attached to motives and Artin representations of $p$-adic Lie extensions. There are two cases of this conjecture: we have the $l$-adic local epsilon factors and a $p$-adic Lie extension, so the two cases are (by abuse of language) $l=p$ and $l \neq p$.

The case $l=p$ seems to be very deep (see Kato \cite{Kato:1993b} and a survey of it, among other results, in Venjakob \cite{Venjakob:2012}) and, as far as the author is aware, very few results are known in this direction. In this paper we prove, up to uniqueness, Kato's local $\epsilon$-conjecture in the $l \neq p$ case. We give a sufficient algebraic condition to ensure uniqueness. Kato's local $\epsilon$-conjecture is known in the commutative case by Yasuda \cite{Yasuda:2009}. In a recent preprint Yasuda \cite{Yasuda:2012} proves the noncommutative version by a method completely different than ours. We remark that we do not use any of the results proven by Yasuda, including the commutative case, in our proof. 

Our strategy for proving the conjecture is as follows: Huber-Kings \cite{HuberKings:2002} (see also Fukaya-Kato \cite{FukayaKato:2006} 2.3.5) made a remarkable observation which roughly says that ``to prove the Equivariant Tamagawa Number Conjecture (ETNC) for arbitrary motives and arbitrary extensions it is enough to prove it for Tate motives and arbitrary extensions". While this strategy has not yet been applied to ETNC except to prove uniqueness (Fukaya-Kato \cite{FukayaKato:2006} 2.3.5), we use it here to prove Kato's local $\epsilon$-conjecture $l \neq p$ case. Therefore we first prove the existence of epsilon constants for Tate motives and then in general by a bootstrapping process. The existence of epsilon constants for Tate motives and finite extensions is equivalent to showing that the Galois Gauss sums of Martinet \cite{Martinet:1977} belong to the image of ``determinant" map from a certain $K_1$ group. This is famously proven for Tame extensions by Taylor \cite{Taylor:1978} during his proof of Fr\"{o}hlich's conjecture. It is subsequently generalised in Holland-Wilson \cite{HollandWilson:1993} to arbitrary extensions using the results of Deligne-Henniart \cite{DeligneHenniart:1981}. We do not use the results of Taylor, Holland-Wilson or Deligne-Henniart in our proof. Our proof has some similarities with Taylor's proof, for instance, both proofs use integral logarithm developed by Oliver and Taylor (in fact we use a more general construction of integral logarithm from Chinburg-Pappas-Taylor \cite{ChinburgPappasTaylor:2012}) and both proofs have certain congruences between Gauss sums. However, the congruences in this paper look different from that in Taylor \cite{Taylor:1978} and we have made no attempt to compare the two.   

Our construction of the $p$-adic epsilon factors is philosophically close to the construction of Langlands but admittedly much easier. Let us explain this briefly. Langlands starts with epsilon constants defined by Tate for one dimensional representation. It is easy to construct a candidate for epsilon constants for arbitrary finite dimensional representations using Brauer's induction theorem. However, it is a very difficult task to prove well-defined-ness of this definition of epsilon factors (see the comment at the end of section \ref{sectionrecall}). Nonetheless, Langlands shows this in his manuscript Langlands \cite{Langlands:1970} by determining presentation of the kernel of ``Brauer's induction map" and then using it to show that the epsilon constants on elements in the kernel are 1. In our proof we use an explicit expression for the $p$-adic epsilon constants for finite abelian extensions. We then prove that the $p$-adic epsilon constant for finite non-abelian extensions exists if the $p$-adic epsilon constants for abelian subextensions satisfy certain compatibilities (or congruences). The precise shape of these compatibilities is determined by a theorem which is in a sense dual to Brauer's induction theorem (see the proof of theorem \ref{theoremk1} in section \ref{sectionappendix}, especially the map $\beta_{G,J}$ which is dual to the induction map). 

The paper is organised as follows. In section \ref{sectionrecall} we briefly review the properties of Deligne-Langlands local epsilon factors after setting up notations for the paper. In section \ref{sectionstatement} we recall the statement of the conjecture. In section \ref{sectionspecialcase}, which is the heart of the paper, we prove the existence of $p$-adic epsilon factors for Tate motives. This uses some algebraic results such as description $K_1$-groups of certain groups rings. These results are collected in the appendix, section \ref{sectionappendix}. Sections \ref{sectiongroupring} and \ref{sectiongeneral} finish the proof of our main theorem \ref{maintheorem}.  

I take this opportunity to thank John Coates for his constant encouragement. I thank David Burns for motivation and stimulating discussions. Without his insistence I would not have taken up this problem.  I also thank Seidai Yasuda for providing a preprint of his paper.

\section{Review of Deligne-Langlands local constants} \label{sectionrecall} In this section we recall basic properties of Deligne-Langlands local constants. We do not recall their definition or construction. Notation from this section is used throughout the paper. Fix a prime number $l$. Let $K$ be a finite extension of $\mathbb{Q}_l$. We denote by $O_K$ the ring of integers of $K$. Let $\pi$ be a uniformiser of $O_K$ and $k=O_K/(\pi)$. Let $|k| = l^d$. Let $\parallel \cdot \parallel$ be the norm on $K$ normalised so that $\parallel \pi \parallel = l^{-d}$.

Fix an algebraic closure $\overline{K}$ of $K$. Let $K^{ur}$ be the maximal unramified extension of $K$ in $\overline{K}$. We denote by $\overline{\mathbb{F}}_l$ the residue field that is an algebraic closure of $\mathbb{F}_l$. Recall that the Weil group $W(\overline{K}/K)$ is defined by 
\[
W(\overline{K}/K) = \{ \sigma \in Gal(\overline{K}/K) : v(\sigma) \in \mathbb{Z}\},
\]
where $v$ is the composition $Gal(\overline{K}/K) \rightarrow Gal(\overline{\mathbb{F}}_l/\mathbb{F}_l) \cong \hat{\mathbb{Z}} = \ilim{n} \mathbb{Z}/n\mathbb{Z}$. Here the isomorphism sends the arithmetic Frobenius $\varphi_l$ in $Gal(\overline{\mathbb{F}}_l/\mathbb{F}_l)$ (the map $x \mapsto x^l$) to $1 \in \mathbb{Z}$. The group $W(\overline{K}/K)$ is regarded as a topological group in the way that the subgroup $I:=Ker(v) = Gal(\overline{K}/K^{ur})$ has the usual topology and is open in $W(\overline{K}/K)$. We denote the wild inertia group by $P$. Consider the following commutative diagram.
\[
\xymatrix{P \ar@{^{(}->}[r] \ar[d] & I \ar@{^{(}->}[r] \ar[d] & W(\overline{K}/K) \ar[d]^{\varphi \mapsto \pi^{-1}} \ar[r] & W(\overline{k}/k) = \mathbb{Z} \ar[d]^{-1} \\
1+(\pi) \ar@{^{(}->}[r] & O^{\times} \ar@{^{(}->}[r] & K^{\times} \ar[r]_{val} & \mathbb{Z}}
\]
The third vertical arrow induces an isomorphism between $W(\overline{K}/K)^{ab}$ and $K^{\times}$. We denote the inverse of this map, the Artin map, by 
\[
rec_K:= rec_{\overline{K}/K} : K^{\times} \xrightarrow{\cong} W(\overline{K}/K)^{ab}.
\]
If $L$ is a Galois extension of $K$, then we put
\[
rec_{L/K} : K^{\times} \rightarrow Gal(L/K)^{ab}
\]
for the Artin map. If the extension $L$ is clear from the context we simply write $rec$ for $rec_{L/K}$. 

Let $E$ be a field containing $\mathbb{Q}_l$ and all $l$-power roots of 1. We endow it with discrete topology. For a topological group $G$, we define $R_E(G)$ to be the ring of continuous finite dimensional virtual representations of $G$ with values in $E$. We tacitly use the term representation to mean actual representation rather than a virtual representation. There is a ring homomorphism 
\[
dim: R_E(G) \rightarrow \mathbb{Z}.
\]
We denote $R_{\mathbb{C}_p}(G)$ simply by $R(G)$. Let $dx$ be a $E$-valued Haar measure on $K$ and let $\psi$ be an additive character of $K$ i.e. a homomorphism $\psi: K \rightarrow E^{\times}$. The Haar measure $dx$ satisfies 
\[
d(ax) = \parallel a \parallel dx \hspace{1cm} \text{ for any } a \in K^{\times}
\]
Define $n(\psi)$ to be the smallest integer $n$ such that $\psi|_{(\pi^{-n})} = 1$. If $\chi$ is a one dimensional character of $W(\overline{K}/K)$ i.e. a homomorphism $\chi : W(\overline{K}/K) \rightarrow E^{\times}$, then we denote the conductor of $\chi \circ rec_K$ by $a(\chi)$ i.e. the smallest non-negative integer $a$ such that $1+(\pi^a)$ is contained in the kernel of $\chi \circ rec_K$. Then for any tuple $(K,V,\psi, dx)$ with $V \in R_E(W(\overline{K}/K))$ Deligne \cite{Deligne:1973} defines 
\[
\epsilon_{0,E}(K,V,\psi,dx) \in E^{\times} \quad \text{ such that }
\]
\begin{enumerate}
\item If 
\[
0 \rightarrow V' \rightarrow V \rightarrow V'' \rightarrow 0
\]
is a short exact sequence of virtual representations in $R_E(W(\overline{K}/K))$, then 
\[
\epsilon_{0,E}(K,V,\psi,dx) = \epsilon_{0,E}(K,V',\psi,dx)\epsilon_{0,E}(K,V'',\psi,dx)
\]
\item For any $a \in K^{\times}$
\[
\epsilon_{0,E}(K,V,\psi_a, dx) = det(V)(rec_K(a))\parallel a \parallel^{-dim(V)} \epsilon_{0,E}(K,V,\psi, dx).
\]
and for any $a \in E^{\times}$
\[
\epsilon_{0,E}(K,V,\psi, adx)  = a^{dim(V)} \epsilon_{0,E}(K,V,\psi,dx),
\]
where $\psi_a(x) := \psi(ax)$. \\
\item We have a functional equation
\[
\epsilon_{0,E}(K,V,\psi,dx) \epsilon_{0,E}(K, V^*(1), \psi_{-1},dx) = 1.
\]
\item If $\int_{O_K}dx =1$ and $ker(\psi)=O_K$, then for an unramified representation $V$ of dimension one 
\[
\epsilon_{0,E}(K,V,\psi, dx) = (-\chi(\varphi^{-1}))^{-1} = -\chi(\varphi),
\]
where $\chi$ gives the action of $W(\overline{K}/K)$ on $V$.
\item Assume that $V$ is a one dimensional representation of $W(\overline{K}/K)$ and the action of $W(\overline{K}/K)$ on $V$ is given by $\chi$. Assume that $\chi$ is ramified, then 
\begin{align*}
\epsilon_{0,E}(K,V,\psi,dx) & = \int_{K^{\times}} \chi^{-1}(x)\psi(x)dx := \sum_{n} \int_{v(x)=n} \chi^{-1}(x)\psi(x)dx \\
& = \int_{c^{-1}O_K^{\times}} \chi^{-1}(x)\psi(x)dx,
\end{align*}
where $c \in K^{\times}$ is any element such that $v(c) = n(\psi) + a(\chi)$. \\
\item Let $L$ be a finite extension of $K$. Put $\psi_L = \psi \circ Tr_{L/K}$ and $dy$ for a Haar measure on $L$. Let $V$ be a virtual representations of $W(\overline{K}/L)$ of dimension 0. Then 
\[
\epsilon_{0,E}(L,V,\psi_L, dy) = \epsilon_{0,E}(K,Ind_{L/K}V, \psi,dx).
\]
More generally if dimension of $V$ is not 0, then Deligne proves existence of $\lambda(L/K, \psi, dx, dy) \in E^{\times}$ such that 
\[
\epsilon_{0,E}(L,V,\psi_L, dy) = \lambda(L/K, \psi \circ Tr_{L/K}, dx, dy)^{-dim(V)} \epsilon_{0,E}(K,Ind_{L/K}V, \psi,dx).
\]
The $\lambda$-invariant is independent of $V$ and depends only on the extension $L/K$, $\psi$ and the choice of Haar measures $dx$ and $dy$ on $K$ and $L$ respectively. Therefore we can compute them by
\[
\lambda(L/K, \psi, dx, dy) = \frac{\epsilon_{0,E}(K, Ind_{L/K}(1), \psi, dx)}{\epsilon_{0,E}(L, 1, \psi \circ Tr_{L/K}, dy)},
\]
where $1$ is the trivial representation of $W(\overline{K}/L)$. This will be useful for us especially when $L/K$ is a finite abelian extension. The explicit formulae for epsilon constants given below show that $\lambda(L/K, \psi, dx,dy)=(-1)^{[L:K]-1}$, when $L/K$ is a finite unramified abelian extension. 
\item If $V$ is an element of $R_E(W(\overline{K}/K))$ and $\chi$ is an unramified one dimensional character of $W(\overline{K}/K)$, then 
\[
\epsilon_{0, E}(K, V\chi, \psi, dx) = \chi(\pi^{sw(V)+dim(V)(n(\psi)+1)}) \epsilon_{0, E}(K, V, \psi, dx),
\] 
where $sw(V)$ is the Swan conductor of $V$ which we do not define here. We only remark that in the case of one dimensional representation $V$ the Swan conductor is 0 for unramified $V$ and it is $a(V)-1$ if $V$ is not unramified. 
\end{enumerate}
We denote $\widehat{\mathbb{Z}_p^{ur}}$, $p$-adic completion of the maximal unramified extension of $\mathbb{Z}_p$, by $J$. From now on we choose and fix a Haar measure $dx_K$ on $K$ such that $\int_{O_K}dx_K =1$. If there is no scpoe for confusion we write $dx$ for $dx_K$. If $V$ is of dimension 0, then (2) implies that $\epsilon_{0,E}$ does not depend on $dx$. In this case we denote the epsilon constant by $\epsilon_{0,E}(K,V, \psi)$. For a character $\chi$ of $W(\overline{K}/K)$ of finite order, the epsilon constant can be computed explicitly as 
\begin{itemize}
\item[(i)] $\epsilon_{0,E}(K,\chi, \psi, dx) = -\chi(\pi^{n(\psi)+1})l^{dn(\psi)}$ if $\chi$ is unramified.
\item[(ii)] $\epsilon_{0,E}(K,\chi, \psi, dx) = l^{dn(\psi)}\sum_{u \in O_K^{\times} (\text{mod } (\pi^{a(\chi)}))} \chi^{-1}(rec(uc^{-1})) \psi_K(uc^{-1})$ if $\chi$ is not unramified.
\end{itemize}

These explicit expressions will be very  useful for us later. If $\rho$ is an Artin representation of $W_K$ of dimension $n_{\rho}$, then a theorem of Brauer (see Deligne \cite{Deligne:1973} 1.10) gives
\[
\rho - n_{\rho} \cdot 1 = \sum_{i} n_i Ind_{K_i/K}(\chi_i -1),
\]
for finite extensions $K_i$ of $K$ and one dimensional representations $\chi_i$ of $W_{K_i}$. Then we have
\[
\epsilon(K,\rho, \psi, dx) = \epsilon(K,1, \psi, dx)^{n_{\rho}} \prod_{i} \epsilon(K_1, \chi_i-1, \psi \circ Tr_{K_i/K})^{n_i}
\]
The main result of Langlands \cite{Langlands:1970} and Deligne \cite{Deligne:1973} is to show that this expression is well-defined.

\section{Statement of the main theorem} \label{sectionstatement} Fix a prime $p$ different from $l$. Let $G$ be a $p$-adic Lie group and $O$ be a ring. We put $\Lambda_O(G) := \ilim{U}O[G/U]$ for the Iwasawa algebra of $G$ with coefficients in $O$; here $U$ runs through open normal subgroups of $G$. Let $T$ be any finitely generated projective $\Lambda_O(G)$-module with a continuous action of $Gal(\overline{K}/K)$ (unless stated otherwise all modules and actions are left).

Recall that we have fixed a Haar measure $dx$ on $K$. We consider the triples $(\Lambda, T, \psi)$, where $\Lambda := \Lambda_O(G)$ for some $p$-adic Lie group $G$ and $O$ is a finite extension of $\mathbb{Z}_p$. For a ring $\Lambda$ we denote by $\tilde{\Lambda}$ the ring $J \hat{\otimes}_{\mathbb{Z}_p} \Lambda$.

\begin{conjecture}[(Kato's local epsilon conjecture: $l \neq p$ case)] There exists a unique way to associate an element
\[
\epsilon_{0,\Lambda}(K,T,\psi) \in K_1(\tilde{\Lambda}),
\]
to each triple $(\Lambda, T, \psi)$ as above, satisfying the following conditions:
\begin{itemize}
\item[(i)] For triples $(\Lambda, T, \psi), (\Lambda, T', \psi), (\Lambda, T'', \psi)$ with common $\Lambda, \psi$ and with an exact sequence $0 \rightarrow T' \rightarrow T \rightarrow T'' \rightarrow 0$, we have
\[
\epsilon_{0,\Lambda}(K, T, \psi) = \epsilon_{0,\Lambda}(K, T', \psi) \epsilon_{0, \Lambda}(K, T'', \psi).
\] 
\item[(ii)] Let $(\Lambda, T, \psi), (\Lambda', T', \psi)$ be triples with a common $\psi$, and let $Y$ be a finitely generated projective $\Lambda'$-module endowed with a continuous right action of $\Lambda$ that is compatible with the action of $\Lambda'$. Assume that $T' \cong Y \otimes_{\Lambda} T$. Then the map $Y \otimes_{\Lambda} : K_1(\tilde{\Lambda}) \rightarrow K_1(\tilde{\Lambda'})$ (for a definition of this map see the example after the definition of $K_1$ groups in section \ref{sectionappendix}) sends $\epsilon_{0, \Lambda}(K, T, \psi)$ to $\epsilon_{0, \Lambda'}(K, T', \psi)$. 
\item[(iii)] Let $c \in O_K^{\times}$, then 
\[
\epsilon_{0,\Lambda}(K, T, c\psi) = [T, rec(c)] \epsilon_{0, \Lambda}(K, T, \psi).
\]
\item[(iv)] We have 
\[
\varphi_p(\epsilon_{0,\Lambda}(K, T, \psi)) = [T, rec(p)] \epsilon_{0,\Lambda}(K, T, \psi),
\]
where $\varphi_p$ is an endomorphism of $K_1(\tilde{\Lambda})$ which is Frobenius on $J$ and identity on $\Lambda$. 
\item[(v)] Let $L$ be a finite extension of $\mathbb{Q}_p$, let $V$ be a finite dimensional $L$-vector space over $L$ endowed with a continuous action of $Gal(\overline{K}/K)$ such that the induced action of $W(\overline{K}/K)$ on $V$ is continuous with respect to the discrete topology of $V$, and let $T$ be an $O_L$-lattice of $V$ that is stable under the action of $Gal(\overline{K}/K)$. Let $E$ be any composite field of $L$ and $\mathbb{Q}_p \otimes_{\mathbb{Z}_p} W(\overline{\mathbb{F}}_p)$ over $\mathbb{Q}_p$. Then $\epsilon_{0,E}(K, V, \psi, dx)$ coincides with the image of $\epsilon_{0,O_L}(K, T, \psi)$. 
\item[(vi)] $\epsilon_{0,\Lambda}(K, T, \psi) = l^{dn(\psi)} [T, -\varphi_l^{1+n(\psi)}]$ is the action of $Gal(\overline{K}/K)$ on $T$ is unramified. 
\end{itemize}
\label{conjecture}
\end{conjecture}

Our main result is the following

\begin{theorem} Conjecture \ref{conjecture} is true up to uniqueness.
\label{maintheorem}
\end{theorem}

\begin{remark} Since we cannot prove uniqueness it is a little awkward to prove properties (i) and (ii) in the conjecture. We contend with proving existence of epsilon constants satisfying (i) and (ii). 
\end{remark}

\section{A special case} \label{sectionspecialcase} First we prove the theorem when $G = Gal(L/K)$ for a $p$-adic Lie extension $L$ of $K$, the ring $O$ is $\mathbb{Z}_p$ and the ring $\Lambda := \Lambda_{\mathbb{Z}_p}(G)$. The module is $T = \Lambda$ as a $\Lambda$-module with the following action of $Gal(\overline{K}/K)$. 
\[
\sigma \cdot t = t \cdot \overline{\sigma}^{-1}.
\]
Here $\overline{\sigma}$ is the image of $\sigma \in Gal(\overline{K}/K)$ in the quotient $G$ and the multiplication on the right hand side is in the ring $\Lambda$. We denote this $\Lambda$-module $T$ by $\Lambda^{\varhash}$. More precisely, we prove the following 

\begin{theorem} There exists $\epsilon_{0, \Lambda}(K, \Lambda^{\varhash}, \psi) \in K_1(\tilde{\Lambda})$ satisfying the following conditions: 
\begin{itemize}
\item[($ii'$)] Let $L' \subset L$ be a Galois extension of $K$ such that $L/L'$ is finite and $G'=Gal(L'/K)$. Put $\Lambda' = \Lambda_{\mathbb{Z}_p}(G')$. Then the homorphism $\Lambda' \otimes_{\Lambda} : K_1(\tilde{\Lambda}) \rightarrow K_1(\tilde{\Lambda'})$ sends $\epsilon_{0, \Lambda}(K,\Lambda^{\varhash}, \psi)$ to $\epsilon_{0, \Lambda'}(K,\Lambda'^{\varhash}, \psi)$.
\item[($iii'$)] If $c \in O_K^{\times}$, then 
\begin{align*}
\epsilon_{0,\Lambda}(K, \Lambda^{\varhash}, c\psi) & = [\Lambda^{\varhash}, rec(c)]\epsilon_{0,\Lambda}(K, \Lambda^{\varhash}, \psi) \\
& = rec(c)^{-1}\epsilon_{0,\Lambda}(K, \Lambda^{\varhash}, \psi).
\end{align*}
\item[($iv'$)] We have
\begin{align*}
\varphi_p(\epsilon_{0,\Lambda}(K, \Lambda^{\varhash}, \psi)) & = [\Lambda^{\varhash}, rec(p)]\epsilon_{0,\Lambda}(K, \Lambda^{\varhash},\psi) \\
& = rec(p)^{-1}\epsilon_{0,\Lambda}(K, \Lambda^{\varhash},\psi)
\end{align*}
\item[($v'$)] Let $\rho$ be a finite dimensional $\mathbb{C}_p$-representation of $G$. We consider it as a representation of $W(\overline{K}/K)$ by first inflating it to $Gal(\overline{K}/K)$ and then restricting it to $W(\overline{K}/K)$. If $\rho$ is a continuous representation of $W(\overline{K}/K)$ (with discrete topology on $\mathbb{C}_p$), then we have
\[
\rho(\epsilon_{0,\Lambda}(K,\Lambda^{\varhash}, \psi)) = \epsilon_{0, \mathbb{C}_p}(K, V_{\rho}, \psi, dx) 
\]
Here $\rho$ on the left hand side is the evaluation map $\rho: K_1(\tilde{\Lambda}) \rightarrow \mathbb{C}_p$ and $V_{\rho}$ is the representation space of $\rho$.
\item[($vi'$)] If $\Lambda^{\varhash}$ is unramified, i.e. if the extension $L/K$ is unramified, then 
\begin{align*}
\epsilon_{0,\Lambda}(K,\Lambda^{\varhash}, \psi) & = l^{dn(\psi)}[T, -\varphi_l^{1+n(\psi)}] \\
& = -l^{dn(\psi)}\varphi^{-1-n(\psi)}
\end{align*}
\end{itemize}
 
\label{theoremspecialcase}
\end{theorem}

Let us remark that when $G$ is a finite group uniqueness follows immediately from $(v')$ and injection of the map 
\[
\prod_{\rho \in R(G)} \rho : K_1(\tilde{\Lambda}) \rightarrow \prod_{\rho} J[\rho]^{\times}
\]
proven in Izychev-Venjakob \cite{IzychevVenjakob:2010}, corollary 2.44. Here $J[\rho]$ is the ring obtained by adjoining values of $\rho$ to $J$. Thus only property $(v')$ is used for uniqueness. In fact, we use this very strongly in several proofs below. For instance, to prove property $(ii')$ it suffices to show that image of $\epsilon_{0, \Lambda}(K, \Lambda^{\varhash}, \psi)$ takes the same values as $\epsilon_{0, \Lambda'}(K, \Lambda'^{\varhash}, \psi)$ at all representations of $G'$. 

\subsection{$G$ is a finite group}

\subsubsection{\bf $G$ is a finite abelian group:} Let $L/K$ be a finite abelian extension with $G= Gal(L/K)$. Let $a = max\{cond(L/K), 1\}$.  For $0 \leq i \leq a-1$, fix elements $c_{a-i}$ of $K^{\times}$ with valuation $a-i+n(\psi)$. For this paragraph define
\begin{align*}
\epsilon_0& :=\epsilon_{0, \mathbb{Z}_p[G]}(K,\mathbb{Z}_p[G]^{\varhash}, \psi) \\
& := l^{dn(\psi)} \Big(\sum_{i=0}^{a-1}\frac{1}{[(\pi^{a-i}):(\pi^a)]} \Big[\sum_{u \in O_K^{\times} (\text{mod }\pi^a)} rec(c_{a-i}u^{-1}) \psi(uc_{a-i}^{-1})\Big] \Big)\\ 
&= \sum_{i=0}^{a-1} l^{d(n(\psi)-i)} \Big(\sum_{v \in O_K^{\times} (\text{mod } \pi^{a-i})} \psi(vc_{a-i}^{-1}) \Big[ \sum_{ u \in v (\text{mod } \pi^a)} rec(c_{a-i}u^{-1}) \Big] \Big)\\
&\in J[G].
\end{align*}
Note that $[(\pi^{a-i}):(\pi^a)] = l^{di}$ for all $0 \leq i \leq a-1$ and if $u \equiv u' (\text{mod } \pi^{a-i})$, then $\psi(c_{a-i}^{-1}u) = \psi(c_{a-i}^{-1}u')$. 

\begin{lemma} Let $L'$ be a Galois extension of $K$ contained in $L$. Put $G' = Gal(L'/L)$. Under the projection 
\[
J[G] \rightarrow J[G'],
\] 
$\epsilon_{0, \mathbb{Z}_p[G]}(K, \mathbb{Z}_p[G]^{\varhash}, \psi)$ is mapped to $\epsilon_{0,\mathbb{Z}_p[G']}(K,\mathbb{Z}_p[G']^{\varhash}, \psi)$.
\label{lemmaprojection}
\end{lemma}
\begin{proof} Let $b := max\{cond(L'/K),1\}$. If $i < a- b$, then the projection of 
\[
\sum_{ u \in O_K^{\times} (\text{mod } \pi^a)} rec_{L/K}(c_{a-i}u^{-1}) \psi(uc_{a-i}^{-1})
\]
in $J[G']$ is
\begin{align*}
&\sum_{v \in O_K^{\times} (\text{mod } \pi^b)} rec_{L'/K}(c_{a-i}v^{-1}) \Big(\sum_{u \in v (\text{mod } \pi^a)} \psi(uc_{a-i}^{-1})\Big) = 0.
\end{align*}
On the other hand if $i \geq a-b$, then the projection of 
\[
\sum_{ u \in O_K^{\times} (\text{mod } \pi^a)} rec_{L/K}(c_{a-i}u^{-1}) \psi(uc_{a-i}^{-1})
\]
in $J[G']$ is
\begin{align*}
& \sum_{v \in O_K^{\times} (\text{mod } \pi^b)} rec_{L'/K}(c_{a-i}v^{-1})\Big(\sum_{u \in v(\text{mod } \pi^a)} \psi(uc_{a-i}^{-1})\Big) \\
= & [(\pi^b):(\pi^a)] \sum_{v \in O_K^{\times} (\text{mod } \pi^b)} rec_{L'/K}(c_{a-i}v^{-1})\psi(vc_{a-i}^{-1}).
\end{align*}
The lemma is now clear.  
\end{proof} 

\begin{lemma} Let $\chi$ be a one dimensional character of $G$ with values in $\mathbb{C}_p$. If the conductor of $\chi$ is $a$ or if $a=1$ and $\chi$ is unramified, then we have
\[
\chi(\epsilon_0) = \epsilon_{0,\mathbb{C}_p}(K, \chi, \psi,dx).
\]
\label{lemmaevaluation}
\end{lemma}
\begin{proof} First assume that $\chi$ has conductor $a$. Then 
\begin{align*}
\chi(\epsilon_0) = & l^{dn(\psi)} \sum_{ u \in O_K^{\times} (\text{mod } \pi^a} \psi(uc_a^{-1})\chi(rec(c_au^{-1}))  \\
+ & \sum_{i=1}^{a-1} l^{d(n(\psi)-i)} \Big(\sum_{v \in O_K^{\times} (\text{mod } \pi^{a-i})} \psi(vc_{a-i}^{-1}) \big[ \sum_{u \in v(\text{mod } \pi^a)} \chi(rec(c_{a-i}u^{-1}))\big]\Big)
\end{align*}
Let $w \in O_K^{\times} / 1 + (\pi^a)$ be any elements whose image in $O_K^{\times}/1+(\pi^{a-i})$ is $v$. Then sum in square brackets above is 
\begin{align*}
\sum_{u \in v (\text{mod } \pi^a)} \chi(rec(c_{a-i}u^{-1})) & = \chi(c_{a-i}w^{-1}) \sum_{u \in 1 (\text{mod } \pi^{a})} \chi(rec(c_{a-i}u^{-1})) \\
\end{align*}
which is zero since $\chi$ is nontrivial on $1+(\pi^{a-i})$ for any $i \geq 1$. 

Next we assume that $\chi$ is unramified. Therefore $a=1$ and the calculations is straightforward.   
\end{proof}

\begin{proposition} Theorem \ref{theoremspecialcase} holds if $G$ is a finite abelian group
\end{proposition}

\begin{proof} We first notice that $\epsilon_0$ lies in $J[G]^{\times} = K_1(J[G])$ (this is similar to the argument in section 4.4 in Taylor \cite{Taylor:1978}). The main point is that $J[G] \cap \mathfrak{M}^{\times} = J[G]^{\times}$, where $\mathfrak{M}$ is the maximal order in $\widehat{\mathbb{Q}_p^{ur}}[G]$ consisting of all $x$ such that $\rho(x)$ is a $p$-adic unit for all $\rho \in R(G)$. Next we show that $\epsilon_0$ satisfies all the conditions in the theorem. Condition $(ii')$ is verified in lemma \ref{lemmaprojection}. Conditions $(iii')$ and $(iv')$ are easily verified. Condition $(v')$ is verified using lemmas \ref{lemmaprojection} and \ref{lemmaevaluation}. The last condition is again easily verified from the definition of $\epsilon_0$.  
 \end{proof}

\begin{remark} If $G$ is of the form $\Delta \times H$, then for any one dimensional character $\rho$ of $\Delta$ we put 
\[
\epsilon_{0,\mathbb{Z}_p[\rho][H]}(K, T\otimes \rho, \psi) = \rho(\epsilon_{0,\mathbb{Z}_p[G]}(K,T,\psi)) \in J[\rho][H]^{\times}
\]
Here $J[\rho]$ is the ring obtained by adjoining the image of $\rho$ to $J$ and $\rho$ on the right hand side is the $J$-linear map 
\[
\rho: J[\Delta \times H] \rightarrow J[\rho][H]
\]
given by $(\delta, h) \mapsto \rho(\delta)h$ for every $(\delta, h) \in \Delta \times H$. Rest of the notation is self evident and we do not explain it further.
\label{remark}
\end{remark}

\subsubsection{\bf $G$ is abelian times a finite group of order prime to $p$} 
\begin{proposition} Theorem \ref{theoremspecialcase} holds when $G$ is of the form $G = \Delta \times H$ where $H$ is a finite abelian group and $\Delta$ is a finite group of order prime to $p$. 
\label{propprimetop}
\end{proposition}

\begin{proof} Use the isomorphism
\begin{equation}
K_1(J[\Delta \times H]) \xrightarrow{\sim} \prod_{\rho \in R(\Delta)} K_1(J[H]) \cong \prod_{\rho \in R(\Delta)} J[H]^{\times},
\label{equationgroupringdecomposition}
\end{equation}
where the map in each component is the one in remark \ref{remark}. The expression $\epsilon_{0,\mathbb{Z}_p[\rho][H]}(K,T \otimes {\rho}, \psi)$ makes sense when $\rho$ is a one dimensional character. If not, then let $n_{\rho}$ be the dimension of $\rho$. We write $\rho - n_{\rho} \cdot 1 = \sum_{i} n_i Ind_{K_i/K}(\chi_i-1)$ and define 
\begin{align*}
& \epsilon_{0,\mathbb{Z}_p[\rho][H]}(K,T\otimes \rho, \psi) := \\
& \epsilon_{0,\mathbb{Z}_p[H]}(K,T\otimes 1, \psi)^{n_{\rho}} \prod_i \epsilon_{0,\mathbb{Z}_p[\chi_i][H]}(K_i, T\otimes (\chi_i-1), \psi \circ Tr_{K_i/K})^{n_i}.
\end{align*}
This is well-defined because the Deligne-Langlands constants are well-defined and the map 
\[
J[H]^{\times} \rightarrow \prod_{\chi \in R(H)} J[\chi]^{\times}
\] 
is injective. We define $\epsilon_{0,\mathbb{Z}_p[G]}(K, J[G], \psi)$ as the inverse image of the element $(\epsilon_{0,\mathbb{Z}_p[\rho][H]}(K,T \otimes \rho, \psi))_{\rho}$ under isomorphism in equation (\ref{equationgroupringdecomposition}). We now verify that $\epsilon_0$ satisfies all the conditions in theorem \ref{theoremspecialcase}. Condition $(v')$ is satisfied by the construction we have given. Condition $(ii')$ is implied by $(v')$. Other conditions are easy to verify. 
\end{proof}

\subsubsection{\bf $G$ is $p$-elementary} The aim of this paragraph is to prove the following 

\begin{proposition} Theorem \ref{theoremspecialcase} holds when $G = H \times \Delta$, where $\Delta$ is a finite cyclic group of order prime to $p$ and $H$ is a finite $p$-group.
\label{proppgroups}
\end{proposition} 

The extension $L^{\Delta}/K$ is tamely ramified. Let $F$ be the maximal unramified extension of $K$ inside $L^{\Delta}$. Then $N:= Gal(L^{\Delta}/F)$ is cyclic normal subgroup of $H$ and the quotient $Gal(F/K)$ is also cyclic. We say that the extension $L^{\Delta}/K$ is split if $Gal(L^{\Delta}/K)$ is a semi direct product of $G/N$ and $N$. Replacing $L^{\Delta}$ by a finite extension, if necessary, we may assume that $L^{\Delta}/K$ is a split extension of $p$-power degree. By taking $L$ to be the compositum of $L^{H}$ and $L^{\Delta}$, we may assume that $G = H \times \Delta$, with $H$ a semi direct product of two cyclic groups, say $N$ and $\Gamma$, of $p$-power order. Subsection \ref{subsectiondescriptionofk1} of the appendix describes $K_1(\tilde{\Lambda})$. We refer to the appendix for all unexplained notation in the remaining of this subsection. 
 
\noindent Let us denote the field $L^{G_i}$ by $K_i$. The key theorem of this section and indeed this article is 
\begin{theorem} $((-1)^{[K_i:K]-1}\epsilon_{0,\mathbb{Z}_p[G_i^{ab}]}(K_i,\mathbb{Z}_p[G_i^{ab}]^{\varhash}, \psi \circ Tr_{K_i/K}))$ satisfies M1-M3.
\label{theorempgroups}
\end{theorem}
\begin{remark}Recall that $\lambda(K_i/K,\psi, dx_K, dx_{K_i}) = (-1)^{[K_i:K]-1}$ as $K_i/K$ is unramified. 
\end{remark}
\begin{proof} For this proof let us denote $(-1)^{[K_i:K]-1}\epsilon_{0, \mathbb{Z}_p[G_i^{ab}]}(K_i,\mathbb{Z}_p[G_i^{ab}]^{\varhash}, \psi \circ Tr_{K_i/K}))$ by $\epsilon_i$ and $\psi \circ Tr_{K_i/K}$ by $\psi_i$. Conditions M1 and M2 follow from formal properties of the epsilon constants. Let us briefly explain this for M1. Let $\rho$ be a representation of $H_j/[H_i,H_i] \times \Delta$. Then evaluation of $\pi_{i,j}(\epsilon_j)$ at $\rho -n_{\rho}$ is 
\[
(\rho - n_{\rho})(\pi_{i,j}(\epsilon_j)) = \epsilon_{0,\mathbb{C}_p}(K_j, \rho-n_{\rho}, \psi_j).
\]
On the other hand, 
\begin{align*}
(\rho -n_{\rho})(Nr_{i,j}(\epsilon_i)) & = Ind(\rho-1)(x_i) \\
& = \epsilon_{0,\mathbb{C}_p}(K_i, Ind(\rho-n_{\rho}), \psi_i) \\
& = \epsilon_{0,\mathbb{C}_p}(K_j, \rho-n_{\rho}, \psi_j) 
\end{align*}
Moreover, the values of $Nr_{i,j}(\epsilon_i)$ and $\pi_{i,j}(\epsilon_j)$ at the trivial representation of $H_j/[H_i,H_i] \times \Delta$ is the same, namely $(l^{d[K_i:K]n(\psi)})^{[K_j:K_i]} = l^{d[K_j:K]n(\psi)}$. Hence $Nr_{i,j}(\epsilon_i) = \pi_{i,j}(\epsilon_j)$. 

Now let us prove M3. Denote the order of the residue field of $K_i$ by $l^{d_i}$. As $K_i/K$ is unramified, $\pi$ is a uniformiser in $K_i$. It is enough to prove that for all $0 \leq k \leq a-1$
\begin{align*}
& \sum_{u \in O_{K_i}^{\times} \text{mod } (\pi^a)} rec(u^{-1}c_{a-k}) \psi_i(uc_{a-k}^{-1}) \\
 \equiv & \sum_{v \in O_{K_{i-1}}^{\times} \text{ mod } (\pi^a)} ver(rec(v^{-1}c_{a-k})) \psi_{i-1}(vc_{a-k}^{-1}) (\text{mod }T_i),
\end{align*}
since $l^{d_i(n(\psi)-k)} = l^{p^id(n(\psi)-k)} \equiv l^{p^{i-1}d(n(\psi)-k)} = l^{d_{i-1}(n(\psi)-k)} (\text{mod } p^i)$ ({\bf Warning:} one has to be slightly careful when $p=2$ and $i=1$ in which case the right hand side in above congruence should be replaced by its minus. But in this case we need to show congruence modulo 2 and hence the minus sign can be ignored). The congruence in the above displayed formula is easily verified as we now show. Let $t$ be the smallest positive integer such that $\Gamma^{p^t}$ acts trivially on $u^{-1}c_{a-k}$. It is clear that $t \leq i$. If $t=i$, then the $\Gamma$-orbit of $rec(u^{-1}c_{a-k})\psi_i(uc_{a-k}^{-1})$ lies in the ideal $T_i$. If $t <i$, then there is unique $w \in O_{K_t}^{\times}/1+(\pi^a)$ such that $w = u^{-1}c_{a-k} (\text{mod } \pi^a)$. Note that $\psi_{i-1}(w^{-1}) \equiv \psi_i(w^{-1}) (\text{mod }p^{i-t})$. Therefore the $\Gamma$-orbit of $rec(w)\psi_i(w^{-1}) - ver(rec(w))\psi_{i-1}(w^{-1})$ lies in the ideal $T_i$ and the congruence is verified. 
\end{proof}  

\begin{proof}[of Proposition \ref{proppgroups}] The above theorem together with theorem \ref{theoremk1} of the appendix gives an element 
\[
\epsilon_0:= \epsilon_{0,\mathbb{Z}_p[G]}(K, \mathbb{Z}_p[G]^{\varhash}, \psi) \in K_1(J[G])
\] 
which is the inverse image of $((-1)^{K_i:K]-1}\epsilon_{0,\mathbb{Z}_p[G_i]}(K_i, \mathbb{Z}_p[G_i^{ab}]^{\varhash}, \psi_i))_i$ under the map $\theta_{G,J}$. Let us verify that $\epsilon_0$ satisfies condition ($v'$). If $\rho$ is an irreducible representation of $G$, then by proposition 25 Serre \cite{Serre:representationtheory}, there exists one dimensional representation $\chi_i$ of $G_i$ such that $\rho = Ind_{G_i}^G\chi_i$. Then 
\begin{align*}
\rho(\epsilon_0) & = \chi_i((-1)^{[K_i:K]-1}\epsilon_{0, \mathbb{Z}_p[G_i^{ab}]}(K_i, \mathbb{Z}_p[G_i^{ab}]^{\varhash}, \psi_i)) \\
& = (-1)^{[K_i:K]-1} \epsilon_{0, \mathbb{C}_p}(K_i, \chi_i, \psi_i, dx_{K_i}) \\
& = \epsilon_{0, \mathbb{C}_p}(K, \rho, \psi, dx),
\end{align*}
since the invariant $\lambda(K_i/K, \psi, dx, dx_{K_i})=(-1)^{[K_i:K]-1}$ as $K_i/K$ is unramified. Condition $(ii')$ is implied by $(v')$. Conditions $(iii')$, $(iv')$ and $(vi')$ are easily verified.    
\end{proof}

\subsubsection{\bf $G$ is $Frac(J)$-elementary} In this paragraph $G$ is a $Frac(J)$-elementary group which we define now.

\begin{definition} Let $q$ be a prime and $\Omega$ be a field. A finite group $U$ is called $q$-$\Omega$-elementary if it is of the form $U = C \rtimes P$, where $P$ is a $q$-group and $C$ is a finite cyclic group of order prime to $q$ and $Im(U \rightarrow Aut(C)) \subset Gal(K(\mu_{|C|})/K) \subset Gal(\mathbb{Q}(\mu_{|C|})/\mathbb{Q}) \cong Aut(C)$. A finite group $U$ is called  $\Omega$-elementary if it is $q$-$\Omega$-elementary for some prime $q$.
\end{definition}

From now on we denote $Frac(J)$ by $\Omega$. Assume that $G$ is $q$-$\Omega$-elementary for some prime $q \neq p$. Let $G = C \rtimes U$, where $U$ is a $q$-group and $C$ is a cyclic group of order prime to $q$. Write $C = C_p \times C'$, where $C_p$ is a $p$-part and $C'$ is a prime to $p$-part. Define $U' = C' \rtimes U$. Therefore we can write $G = C_p \rtimes U'$. Put $\overline{C} := H_0(U',C_p)$ and $U_1 := Ker(U' \rightarrow Aut(C_p))$. Let $G_0 = \overline{C} \times U'$ and $G_1 = C_p \times U_1$
\[
\theta_G: K_1(J[G]) \rightarrow K_1(J[G_0]) \times K_1(J[G_1])
\]
which is the natural projection in the first factor and the norm map in the second factor. Then it can be proven (using exactly the same proof as that of proposition 4.11 Kakde \cite{Kakde:2012}) that this map is injective and the image consists of all $(x_0,x_1)$ satisfying
\begin{itemize}
\item[M1.] $N(x_0) = \pi(x_1)$, where $N$ is the norm map 
\[
N : K_1(J[G_0]) \rightarrow K_1(J[\overline{C} \times U_1]),
\]
and $\pi$ is the natural projection 
\[
\pi : K_1(J[G_1]) \rightarrow K_1(J[\overline{C} \times U_1])
\]
\item[M2.] $x_1$ is fixed under conjugation by every element of $U'$.
\end{itemize}
Let us denote $\mathbb{Z}_p[G_0]$ by $\Lambda_0$ and $\mathbb{Z}_p[G_1]$ by $\Lambda_1$. Put $K_1 = L^{G_1}$. Note that proposition \ref{propprimetop} gives elements $\epsilon_{0, \Lambda_0}(K, \Lambda_0^{\varhash}, \psi) \in K_1(J[G_0])$ and $\epsilon_{0,\Lambda_1}(K_1, \Lambda_1^{\varhash}, \psi \circ Tr_{K_1/K}) \in K_1(J[G_1])$. We put $\lambda = \lambda(K_1/K, \psi, dx_K, dx_{K_1})$. As $p \ndivides [G:G_1]$, we may consider $\lambda$ as an elements of $J^{\times}$ using the explicit formulae for $\lambda$-invariants and $\epsilon$-constants given in section \ref{sectionrecall}.

\begin{theorem} The pair $(\epsilon_{0,\Lambda_0}(K,\Lambda_0^{\varhash}, \psi), \lambda \epsilon_{0,\Lambda_1}(K_1,\Lambda_1^{\varhash}, \psi\circ Tr_{K_1/K}))$ lies in the image of $\theta_G$.  
\label{theoremqelementary}
\end{theorem}
\begin{proof} Similar to the first part of the proof of theorem \ref{theorempgroups}.  
\end{proof}

Now assume that $G$ is $p$-$\Omega$-elementary. Write $G = C_n \times U$ (since the group $Gal(\Omega(\mu_n)/\Omega)$ is trivial we have direct product),  where $U$ is a $p$-group and $C_n$ is a cyclic group of order $n$ prime to $p$. Then theorem \ref{theoremspecialcase} is known to be true from proposition \ref{proppgroups}.

\begin{proposition} Theorem \ref{theoremspecialcase} holds in the case when $G$ is a finite $\Omega$-elementary group.
\label{propqpelementary}
\end{proposition}
\begin{proof} We only need consider the case of $q$-$\Omega$-elementary, for $p \neq q$. Then by theorem \ref{theoremqelementary} there is an element $\epsilon_0 = \epsilon_{0, \Lambda}(K, \Lambda^{\varhash}, \psi)$ in $K_1(J[G])$ which maps to $(\epsilon_{0,\Lambda_0}(K,\Lambda_0^{\varhash}, \psi), \lambda \epsilon_{0,\Lambda_1}(K_1,\Lambda_1^{\varhash}, \psi\circ Tr_{K_1/K}))$ under the map $\theta_G$. Rest of the proof is similar to the proof of proposition \ref{proppgroups}. We just remark that to prove property $(v')$ we again use that any irreducible representation is either inflated from a representation of $G_0$ or is induced from an irreducible representation of $G_1$ (proposition 25 in Serre \cite{Serre:representationtheory}).  
\end{proof}

\subsubsection{\bf $G$ is any finite group} In this paragraph $G$ is any finite group. Then the induction theorems of Dress and Wall (see Dress \cite{Dress:1973}, Oliver \cite{Oliver:1988} theorem 11.2 or Wall \cite{Wall:1974} theorem 5.4) give an isomorphism
\[
K_1(J[G]) \xrightarrow{\sim} \ilim{U} K_1(J[U]),
\]
where $U$ ranges over all $\Omega$-elementary subgroups of $G$ and the inverse limit is with respect to the norm maps and the conjugation maps. We have already constructed $\epsilon_U:=\epsilon_{0,\mathbb{Z}_p[U]}(L^U, \mathbb{Z}_p[U]^{\varhash}, \psi \circ Tr_{L^U/K})$ for each $U$ above. Put $\lambda_U:= \lambda(L^U/K, \psi, dx_K, dx_{L^U})$. We contend that $(\lambda_U\epsilon_U)_U \in \ilim{U} K_1(J[U])$. Let $V=gUg^{-1}$ for $g \in G$ and let $\rho$ be any irreducible representation of $V$. Then
\begin{align*}
\rho(\lambda_V\epsilon_V) & = \lambda_V^{dim(\rho)}\epsilon_{0, \mathbb{C}_p}(L^V, \rho, \psi \circ Tr_{L^V/K}, dx_{L^V}) \\
& = \epsilon_{0, \mathbb{C}_p}(K, Ind_{L^V/K}(\rho), \psi, dx).
\end{align*}
On the other hand, 
\begin{align*}
\rho(g\lambda_U \epsilon_U g^{-1}) & = g^{-1}\rho g (\lambda_U\epsilon_U) \\
& = \lambda_U^{dim(\rho)} \epsilon_{0, \mathbb{C}_p}(L^U, g^{-1}\rho g, \psi \circ Tr_{L^U/K}, dx_{L^U}) \\
& = \epsilon_{0, \mathbb{C}_p}(K, Ind_{L^U/K}(g^{-1} \rho g), \psi, dx) \\
& = \epsilon_{0, \mathbb{C}_p}(K, Ind_{L^V/K}(\rho), \psi, dx),
\end{align*}
where the last equality follows from the fact that $Ind_{L^V/K} (\rho) \cong Ind_{L^U/K} (g^{-1}\rho g)$. Hence $\lambda_V\epsilon_V = g\lambda_U\epsilon_U g^{-1}$. Next let $V$ be a $\Omega$-elementary subgroup of another $\Omega$-elementary subgroup $U$ of $G$. Then we must prove that the norm map $K_1(J[U]) \xrightarrow{Nr} K_1(J[V])$ sends $\lambda_U\epsilon_U$ to $\lambda_V\epsilon_V$. Let $\rho$ be any representation of $V$. Then 
\begin{align*}
\rho(\lambda_V\epsilon_V) & = \lambda_V^{dim(\rho)}\epsilon_{0, \mathbb{C}_p}(L^V, \rho, \psi \circ Tr_{L^V/K}, dx_{L^V}) \\
& = \epsilon_{0, \mathbb{C}_p}(K, Ind_{L^V/K}(\rho), \psi, dx)
\end{align*}
On the other hand
\begin{align*}
\rho(Nr(\lambda_U\epsilon_U)) & = Ind_{L^V/L^U}(\rho)(\lambda_U\epsilon_U) \\
& = \lambda_U^{[L^V:L^U]dim(\rho)} \epsilon_{0,\mathbb{C}_p}(L^U, Ind_{L^V/L^U}(\rho), \psi \circ Tr_{L^U/K}, dx_{L^U}) \\
& = \epsilon_{0, \mathbb{C}_p}(K, Ind_{L^U/K}(Ind_{L^V/L^U}(\rho)), \psi, dx) \\
& = \epsilon_{0, \mathbb{C}_p}(K, Ind_{L^V/K}(\rho), \psi, dx)
\end{align*}
Therefore, $Nr(\lambda_U\epsilon_U) = \lambda_V\epsilon_V$. Hence $(\lambda_U\epsilon_U)_U \in \ilim{U} K_1(J[U])$. This allows us to construct $\epsilon_{0,\mathbb{Z}_p[G]}(K,\mathbb{Z}_p[G]^{\varhash}, \psi)$ as the inverse image of $(\epsilon_U)_U$ under above isomorphism. 

\begin{proposition} Theorem \ref{theoremspecialcase} holds in the case when $G$ is any finite group.
\label{propfinitegroup}
\end{proposition}
\begin{proof} We must show that the element $\epsilon_{0, \mathbb{Z}_p[G]}(K, \mathbb{Z}_p[G]^{\varhash}, \psi)$ constructed above satisfies all the conditions required by theorem \ref{theoremspecialcase}. The proof is similar to the proof of proposition \ref{proppgroups}. We just remark that instead of propositon 25 in \cite{Serre:representationtheory} we use Brauer's induction theorem (see for example theorem 19 in Serre \cite{Serre:representationtheory}) which ensures that every irreducible representation of $G$ is a $\mathbb{Z}$-linear combination of representations induced from representations of elementary subgroups. 

\end{proof}

\subsection{$G$ is a $p$-adic Lie group} In this case we use the isomorphism of Fukaya-Kato \cite{FukayaKato:2006} proposition 1.5.1
\[
K_1(\Lambda(G)) \cong \ilim{U} K_1(\mathbb{Z}_p[G/U]),
\]
where $U$ runs through open normal subgroups of $G$. We define $K_1'(\mathbb{Z}_p[G/U]):=K_1(\mathbb{Z}_p[G/U])/SK_1(\mathbb{Z}_p[G/U])$, where $SK_1(\mathbb{Z}_p[G/U]):= ker(K_1(\mathbb{Z}_p[G/U]) \rightarrow K_1(\mathbb{Q}_p[G/U]))$. Define $K_1'(\Lambda(G)) := K_1(\Lambda(G))/\ilim{U}SK_1(\mathbb{Z}_p[G/U])$. Consider the diagram
\[
\xymatrix{
1 \ar[r] & \ilim{U} K_1'(\mathbb{Z}_p[G/U]) \ar[r] & \ilim{U}K_1(J[G/U]) \ar[r]^{1 - \varphi_p} & \ilim{U} K_1(J[G/U]) & \\
& K_1'(\Lambda(G)) \ar[u]^{\sim} \ar[r]_{\iota} & K_1(\Lambda_J(G)) \ar[r]_{1 - \varphi_p} \ar[u] & K_1(\Lambda_J(G)) \ar[u] \ar[r] & 1
}
\]
The bottom row is not known to be exact though it is expected to be (see remark 3.4.6 in \emph{loc. cit.} and Izychev-Venjakob \cite{IzychevVenjakob:2010}). We do not know if the second vertical arrow is an isomorphism. Hence the proof of the following proposition uses property (for the first time) $(iv')$. 

\begin{proposition} Theorem \ref{theoremspecialcase} holds when $G$ is any $p$-adic Lie group.
\label{proppadic}
\end{proposition}
\noindent{\bf Proof:} The surjection of $1-\varphi_p$ in the bottom row (see proposition 3.4.5 Fukaya-Kato \cite{FukayaKato:2006}) implies that there is an element $f \in K_1(\Lambda_J(G))$ such that $(1-\varphi_p)(f) = [\Lambda(G)^{\varhash}, rec(p)]$. Denote the image of $f$ under the second vertical arrow by $(f_U)_U$. Then $(1-\varphi_p)(f_U) = [\mathbb{Z}_p[G/U]^{\varhash}, rec(p)]$. Hence 
\[
(\epsilon_{0, \mathbb{Z}_p[G/U]}(K, \mathbb{Z}_p[G/U]^{\varhash}, \psi) f_U^{-1})_U \in \ilim{U} K_1'(\mathbb{Z}_p[G/U]).
\]
Let $u$ be the inverse image of this element in $K_1'(\Lambda(G))$. Define 
\[
\epsilon_{0, \Lambda(G)}(K, \Lambda(G)^{\varhash}, \psi) = \iota(u)f \in K_1(\Lambda_J(G)).
\]
The proof that this element satisfies all the required properties is similar (in fact easier) to the proof at the end of section \ref{lastsection}. We only make the following remark about property $(v')$. If $\rho$ is a representation of $G$ factoring through a finite quotient $G/U$, then $(v')$ can be verified directly using $(v')$ for the element $\epsilon_{0, \mathbb{Z}_p[G/U]}(K, \mathbb{Z}_p[G/U]^{\varhash}, \psi)$. More generally we use the fact that any finite dimensional continuous representation of $W(\overline{K}/K)$ can be twisted by one dimensional unramified representation to get a representation that factors though a finite quotient of $W(\overline{K}/K)$ and hence through a finite quotient of $G$. We can then use property (7) of Deligne-Langlands constants mentioned in section \ref{sectionrecall}.

\section{$\Lambda$ is group ring of a finite group and $T$ is arbitrary} \label{sectiongroupring} In this section $\Lambda = O[P]$ is a group ring of a finite group and $T$ is arbitrary finitely generated projective $\Lambda$-module. Therefore $T$ is compact and the action of $Gal(\overline{K}/K)$ factors through a p-adic Lie quotient, say $G$. We define $\epsilon_{0, \Lambda}(K, T, \psi)$ as follows: Let $Y$ be the $\Lambda$-module $T$ on which $\Lambda(G)$ acts on the right by 
\[
y \cdot \sigma = \tilde{\sigma}^{-1} \cdot y \hspace{1cm} \text{ for } y \in Y, \sigma \in G.
\] 
Here $\tilde{\sigma}$ is any lift of $\sigma$ to $Gal(\overline{K}/K)$ and the action on the right hand side is the original Galois action on $T$. Then 
\[
Y \otimes_{\Lambda(G)} \Lambda(G)^{\varhash} \cong T 
\] 
as $\Lambda$-modules with the action of $Gal(\overline{K}/K)$. We define $\epsilon_{0, \Lambda}(K, T, \psi)$ to be the image of $\epsilon_{0, \Lambda(G)}(K , \Lambda(G)^{\varhash} , \psi)$ under the map 
\[
Y \otimes_{\Lambda(G)} : K_1(\Lambda(G)) \rightarrow K_1(\Lambda).
\]
Note that this definition is enforced upon us by condition (ii) in the conjecture. Moreover, condition $(ii')$ in theorem \ref{theoremspecialcase} implies that this definition is independent of the choice of $G$. 

\begin{theorem} The element $\epsilon_{0,\Lambda}(K,T,\psi_K)$ defined above is the unique element in $K_1(\tilde{\Lambda})$ satisfying (i)-(vi) of conjecture \ref{conjecture}.
\label{theoremfinitegrouprings}
\end{theorem}
\noindent{\bf Proof:} Again the uniqueness follows from the fact that the map 
\[
K_1(J\otimes_{\mathbb{Z}_p} O[P]) \rightarrow \prod_{\rho \in R(P)} (J \otimes_{\mathbb{Z}_p}O[\rho])^{\times}
\] 
is injective (see corollary 2.44 Izychev-Venjakob \cite{IzychevVenjakob:2010}). Now we show that it has the required properties. Assume that we have three finitely generated projective $\Lambda$-modules $T_1, T_2, T_3$ in a short exact sequence $T_1 \hookrightarrow T_2 \twoheadrightarrow T_3$. We choose $G$ such that the Galois action on all $T_i$'s factors through $G$. Then $\epsilon_{0, \Lambda}(K, T_i, \psi)$ is the image of $\epsilon:= \epsilon_{0, \Lambda(G)}(K, \Lambda(G)^{\varhash}, \psi)$ under the map $\pi_i: Y_i \otimes_{\Lambda(G)}$ (for $1 \leq i \leq 3$). Note that we have a short exact sequence $Y_1 \hookrightarrow Y_2 \twoheadrightarrow Y_3$. Therefore, $\pi_2(a) = \pi_1(a)\pi_3(a)$. Hence 
\[
\epsilon_{0,\Lambda}(K, T_2, \psi) = \pi_2(\epsilon) = \pi_1(\epsilon)\pi_3(\epsilon) = \epsilon_{0, \Lambda}(K, T_1, \psi) \epsilon_{0, \Lambda}(K, T_3, \psi),
\]
i.e. (i) is satisfied.

\noindent We next verify condition (ii). We keep the notation from (ii) in theorem \ref{maintheorem}. Choose $G$ such that the action of Galois on $T$ and $T'$ both factor through $G$. Then $\epsilon_{0, \Lambda}(K, T, \psi)$ and $\epsilon_{0, \Lambda'}(K, T', \psi)$ are images of $\epsilon:=\epsilon_{0, \Lambda(G)}(K, \Lambda(G)^{\varhash}, \psi)$ under the maps, say $\pi_1 := Y_1 \otimes_{\Lambda(G)}$ and $\pi_2:= Y_2 \otimes_{\Lambda(G)}$ respectively. Then 
\[
Y \otimes_{\Lambda} Y_1 \otimes_{\Lambda(G)} \Lambda(G)^{\varhash} \cong Y \otimes_{\Lambda} T \cong T'
\]
Therefore the image of $\epsilon$ under $\pi_2$ and $Y\otimes_{\Lambda} \pi_1$ is the same. Hence (ii) is satisfied. 

\noindent Conditions (iii) and (iv) are easily verified. Condition (v) is verified from condition $(v')$ for $\epsilon_{0, \Lambda(G)}(K, \Lambda(G)^{\varhash}, \psi)$. Condition (vi) holds because of condition $(vi')$ holds for $\epsilon_{0, \Lambda(G)}(K, \Lambda(G)^{\varhash}, \psi)$.

\section{The general case} \label{sectiongeneral} \label{lastsection} In this section $\Lambda:= \Lambda_O(G)$ is Iwasawa algebra and $T$ is arbitrary finitely generated projective $\Lambda$-module. We again use the isomorphism of Fukaya-Kato \cite{FukayaKato:2006} proposition 1.5.6
\[
K_1(\Lambda) \xrightarrow{\sim} \ilim{U} K_1(O[G/U]),
\]
where $U$ runs through open normal subgroup of $G$. Theorem \ref{theoremfinitegrouprings} gives existence of $\epsilon_U:=\epsilon_{0, O[G/U]}(K, O[G/U] \otimes_{\Lambda_O(G)}T, \psi) \in K_1(O[G/U])$ for each U. Property (ii) for these elements along with uniqueness ensures that the tuple $(\epsilon_U)_U$ lies in the group $\ilim{U}K_1(\widetilde{O[G/U]})$. Now consider the following commutative diagram
\[
\xymatrix{
1 \ar[r] & \ilim{U} K_1'(O[G/U]) \ar[r] & \ilim{U}K_1(\widetilde{O[G/U]}) \ar[r]^{1 - \varphi_p} & \ilim{U} K_1(\widetilde{O[G/U]}) & \\
& K_1'(\Lambda) \ar[u]^{\sim} \ar[r]_{\iota} & K_1(\tilde{\Lambda}) \ar[r]_{1 - \varphi_p} \ar[u] & K_1(\tilde{\Lambda}) \ar[u] \ar[r] & 1
}
\]
The top row is exact by theorem 2.45 Izychev-Vanjakob \cite{IzychevVenjakob:2010}. The bottom row is not known to be exact though it is expected to be. Consider $[T, rec(p)]$ as an element of $K_1(\tilde{\Lambda})$. Its image under the second vertical arrow is $([T_U, rec(p)])_U$, where $T_U = O[G/U] \otimes_{\Lambda_O(G)} T$. By proposition 3.4.5 in Fukaya-Kato \cite{FukayaKato:2006} the map $1-\varphi_p$ in the bottom row is sujective and hence there is an element $f \in K_1(\tilde{\Lambda})$ such that $\varphi_p(f) = [T,rec(p)]f$. Denote the image of $f$ under the second vertical arrow by $(f_U)_U$. Property (iv) of the epsilon constants $\epsilon_U$ and their uniqueness implies that $(\epsilon_U f_U^{-1})_U \in \ilim{U} K_1'(O[G/U])$. Let $u:=u(T)$ be the inverse image of this element under the first vertical arrow. Define $\epsilon_{0,\Lambda}(K, T, \psi) := \iota(u)f$. This is the element we seek. Now we can finish the proof of our main theorem.

\begin{proof}[of theorem \ref{maintheorem}] We show that the element $\epsilon_{0, \Lambda}(K, T, \psi)$ defined above satisfies all the required properties in theorem \ref{maintheorem}. Let $T, T', T''$ be as in (i) of theorem \ref{maintheorem}. Then, for every open normal subgroup $U$ of $G$, the modules $T_U, T'_U, T''_U$ are related in the same manner (since they are projective $\Lambda$-modules). We then have the relation 
\[
\epsilon_{0, O[G/U]}(K, T_U, \psi) = \epsilon_{0, O[G/U]}(K, T'_U, \psi) \epsilon_{0,O[G/U]}(K, T''_U, \psi).
\]
Moreover, 
\[
[T, rec(p)] = [T', rec(p)][T'', rec(p)].
\]
Hence $u(T) = u(T')u(T'')$ and condition (i) is satisfied. 

\noindent We next verify condition (ii). For an open normal subgroup $U$ of $G$ and an open normal subgroup $U'$ of $G'$, we put $Y_{U',U}$ for the $O[G'/U']-O[G/U]$-bimodule $O[G'/U'] \otimes_{\Lambda(G')} Y \otimes_{\Lambda(G)} O[G/U]$. Then we have
\[
Y_{U',U} \otimes_{O[G/U]} T_U \cong T'_{U'}
\]
Hence, under the map $Y_{U',U} \otimes_{O[G/U]}$, the element $\epsilon_{0, O[G/U]}(K, T_U, \psi)$ maps to $\epsilon_{0, O[G'/U']}(K, T_{U'}, \psi)$. Taking into account that 
\[
Y \otimes_{\Lambda} [T, rec(p)] = [T', rec(p)] \in K_1(\tilde{\Lambda}')
\]
we get (ii). In order to verify (iii) we notice that $f$ used in the construction of $\epsilon_{0, \Lambda}(K, T, \psi)$ does not depend on $\psi$. The factors $\epsilon_U$, which depend on $\psi$, satisfy (iii). Therefore $u(T)$, which depends on $\psi$, satisfy (iii). Hence we get (iii) for $\epsilon_{0, \Lambda}(K, T, \psi)$. Condition (iv) holds by construction. Condition (vi) is easily verified.  
\end{proof}

We end with an algebraic criterion to prove uniqueness of the epsilon factors.

\begin{proposition} If for any compact $p$-adic Lie group $G$ and any finite extension $O$ of $\mathbb{Z}_p$, we have an injection 
\[
K_1(\widetilde{\Lambda_O(G)}) \hookrightarrow \ilim{n} K_1(\widetilde{\Lambda_O(G)}/I^n),
\]
where $I$ is the radical ideal of $\widetilde{\Lambda_O(G)}$, then the uniqueness in Conjecture \ref{conjecture} holds.
\label{propsuffcondition}
\end{proposition}
\begin{proof} This is clear from our construction.  
\end{proof}

\section{Appendix: $K_1$ of group rings} \label{sectionappendix} In this section we gather algebraic results used in the main text of the article. 

\subsection{Definition of $K_1$} \begin{definition} Let $R$ be a ring with 1. The group $K_1(R)$ is an abelian group, whose group law we denote multiplicatively, defined by the following generators and relations. 
\begin{itemize}
\item[] Generators: $[P, \alpha]$, where $P$ is a finitely generated projective $R$-module and $\alpha$ is an automorphism of $P$. 
\item[] Relations: 
\begin{itemize}
\item[(i)] $[P, \alpha] = [Q, \beta]$ if there is an isomorphism $f$ from $P$ to $Q$ such that $f \circ \alpha = \beta \circ f$.
\item[(ii)] $[P, \alpha \circ \beta] = [P, \alpha] [P, \beta]$.
\item[(iii)] $[P\oplus Q, \alpha \oplus \beta] = [P, \alpha] [Q, \beta]$.
\end{itemize}
\end{itemize}
\end{definition}

\begin{example}(Change of rings) Let $R$ and $R'$ be two rings. Let $Y$ be a finitely generated projective $R'$-module. Assume that $Y$ is also a right $R$-module and actions of $R$ and $R'$ are compatible i.e. $Y$ is a $R'$-$R$ bimodule. Then there is a map 
\[
Y \otimes_R :K_1(R) \rightarrow K_1(R')
\]
given by $[P,\alpha] \mapsto [Y \otimes_R P, id_Y \otimes \alpha]$. 
\end{example}

Here is an alternate description of $K_1(R)$. We have a canonical homomorphism $GL_n(R) \rightarrow K_1(R)$ defined by mapping $\alpha$ in $GL_n(R)$ to $[R^n, \alpha]$, where $R^n$ is regarded as a set of row vectors and $\alpha$ acts on them from the right. Now using the inclusion maps $GL_n(R) \hookrightarrow GL_{n+1}(R)$ given by $g \mapsto \left( \begin{array}{cc} g & 0 \\ 0 & 1\end{array} \right)$, we let 
\[
GL(R) = \cup_{n \geq 1} GL_n(R).
\]
Then the homomorphisms $GL_n(R) \rightarrow K_1(R)$ induce an isomorphism (see for example Oliver \cite{Oliver:1988}, chapter 1)
\[
\frac{GL(R)}{[GL(R), GL(R)]} \xrightarrow{\sim} K_1(R),
\]
where $[GL(R), GL(R)]$ is the commutator subgroup of $GL(R)$. If $R$ is commutative, then the determinant maps, $GL_n(R) \rightarrow R^{\times}$, induce the determinant map
\[
det: K_1(R) \rightarrow R^{\times}, 
\]
via the above isomorphism. This gives a splitting of the canonical homomorphism $R^{\times} = GL_1(R) \rightarrow K_1(R)$. If $R$ is semilocal then Vaserstein (\cite{Vaserstein:1969} and \cite{Vaserstein:2005}) proves that the canonical homomorphism $R^{\times} =GL_1(R) \rightarrow K_1(R)$ is surjective. From these two facts we conclude that if $R$ is a semilocal commutative ring, then the determinant map induces a group isomorphism between $K_1(R)$ and $R^{\times}$.

\begin{example} If $P$ is a compact $p$-adic Lie group and $O$ is a complete discrete valuation ring of characteristic 0 with residue field of characteristic $p$, then the Iwasawa algebra $\Lambda_O(P)$ is a semilocal ring. Hence, if $P$ is an abelian compact $p$-adic Lie group then $K_1(\Lambda_O(P)) \cong \Lambda_O(P)^{\times}$. 
\end{example}
 
Let $I \subset R$ be any ideal. Denote the group of invertible matrices in $GL(R)$ which are congruent to the identity modulo $I$ by $GL(R,I)$. Denote the smallest normal subgroup of $GL(R)$ containing all elementary matrices congruent to the identity modulo $I$ by $E(R, I)$. Finally, set $K_1(R,I) = GL(R,I)/E(R,I)$. A lemma of Whitehead says that $E(R,I) = [GL(R), GL(R,I)]$. Hence $K_1(R,I)$ is an abelian group.

\subsection{$K_1$ of some group rings} \label{subsectiondescriptionofk1} Let $N$ and $\Gamma$ be finite cyclic groups of $p$-power orders and $H= N \rtimes \Gamma$. Let $\Delta$ be a finite cyclic group whose order is prime to $p$ and put $G=H \times \Delta$.  In this section we describe $K_1(J[G])$ following Kato \cite{Kato:2005} and Kakde \cite{Kakde:2011}. Let us set up some notation.  Assume that the order of $\Gamma$ is $p^n$. We put
\[
H_i:= N \rtimes \Gamma^{p^i} \hspace{.5cm} \text{and} \hspace{.5cm} G_i := H_i \times \Delta \hspace{1cm} \text{for all }  0 \leq i \leq n.
\] 
Therefore, $H_0 =H$ and $H_n = N$. We define the map 
\[
\theta_{G, J} : K_1(J[G]) \rightarrow \prod_{i=0}^n J[G_i^{ab}]^{\times},
\]
which is, in each factor, a composition of the norm map,
\[
norm: K_1(J[G]) \rightarrow K_1(J[G_i]), 
\]
followed by the natural projection
\[
K_1(J[G_i]) \rightarrow K_1(J[G_i^{ab}]) \cong J[G_i^{ab}]^{\times}.
\]
We define some more maps. For all $0 \leq i \leq j \leq n$, there are two maps; the norm map
\[
Nr_{i,j} : J[G_i^{ab}]^{\times} \rightarrow J[H_j/[H_i, H_i] \times \Delta]^{\times},
\] 
and the natural surjection
\[
\pi_{i,j}: J[G_j^{ab}] \rightarrow J[H_j/[H_i,H_i] \times \Delta].
\]
For every $0 \leq i \leq n$, we have 
\[
\sigma_i : J[G_i^{ab}] \rightarrow J[G_i^{ab}]^{\Gamma}
\]
given by 
\[
x \mapsto \sum_{\gamma \in \Gamma/\Gamma^{p^i}} \gamma x \gamma^{-1}.
\]
Put $T_i$ for the image of $\sigma_i$. Then $T_i$ is an ideal in the ring $J[G_i^{ab}]^{\Gamma}$. Lastly, for every $1 \leq i \leq n$, we need the map 
\[
ver_i : J[G_{i-1}^{ab}] \rightarrow J[G_i^{ab}],
\] 
induced by the Frobenius on the coefficients $J$ and the transfer map $ver : G_{i-1}^{ab} \rightarrow G_i^{ab}$.

\begin{theorem} The map $\theta_{G,J}$ is injective and its image consists of tuples $(x_i)_{i=0}^n$ such that 
\begin{itemize}
\item[M1.] For all $0 \leq i \leq j \leq n$, we have $Nr_{i,j}(x_i) = \pi_{i,j}(x_j)$. 
\item[M2.] For all $0 \leq i \leq n$, the element $x_i$ is fixed under conjugation by every $g \in G$. 
\item[M3.] For all $1 \leq i \leq n$, we have the congruence $x_{i} \equiv ver(x_{i-1}) (\text{mod } T_i)$. 
\end{itemize} 
\label{theoremk1}
\end{theorem}

\begin{proof} First we state an additive result. Let $Conj(G)$ denote the set of conjugacy classes of $G$ and define $\tau$ to be the natural surjection $\tau: J[G] \rightarrow J[Conj(G)]$. Put $I_G = Ker(J[G] \rightarrow J)$. Define 
\[
J[Conj(G)] \xrightarrow{\beta_{G,J}} \prod_{i=0}^n J[G_i^{ab}],
\]
where the map in each component is the composition of the trace map
\[
J[Conj(G)] \rightarrow J[Conj(G_i)]
\]
\[
\text{$J$-linear and } \hspace{1cm} g \mapsto \left\{ \begin{array}{c c} 0 & \text{ if $g \notin G_i$} \\ \sum_{x \in G/G_i} xgx^{-1} & \text{ if $g \in G_i$} \end{array} \right.
\]
with the natural surjection $J[Conj(G_i)] \rightarrow J[G_i^{ab}]$. The map $\beta_{G,J}$ is injective and the image consists of all $(a_i)_{1 \leq i \leq n} \in \prod_{i=0}^n J[G_i^{ab}]$ such that
\begin{itemize}
\item[A1.] For all $0 \leq i \leq j \leq n$, we have $Tr_{i,j}(a_i) = \pi_{i,j}(a_j)$. Here $Tr_{i,j}$ is the trace map
\[
Tr_{i,j} : J[G_i^{ab}] \rightarrow J[H_j/[H_i, H_i] \times \Delta].
\]
\item[A2.] For all $0 \leq i \leq n$, the element $a_i$ is fixed under conjugation by every $g \in G$. 
\item[A3.] For all $0 \leq i \leq n$, we have $a_i \in T_i$. 
\end{itemize}
Moreover, the image of $\tau(I_G)$ under $\beta_{G,J}$ is precisely equal to the elements in $\prod_{i=0}^n I_{H_i^{ab}}$ satisfying A1-A3. The last statement is clear from the remaining claims whose proof can be found in Kakde \cite{Kakde:2011} section 3.1, or Kakde \cite{Kakde:2012} section 5.3 or Schneider-Venjakob \cite{SchneiderVenjakob:2012} section 3. Next we use the integral logarithm of Oliver and Taylor. This was classically defined only with the ring of integers in finite unramified extension of $\mathbb{Q}_p$ as coefficients. Hence we use the recent extension to more general coefficients as in Chinburg-Pappas-Taylor \cite{ChinburgPappasTaylor:2012}.  Define $K_1'(J[G], I_G)$ to be the image of $K_1(J[G], I_G)$ in $K_1(J[G])$ under the natural map $K_1(J[G], I_G) \rightarrow K_1(J[G])$. Combining equation (14) of \emph{loc. cit.} with the fact that $SK_1(J[G]) = \{1\}$ (corollary 2.44 in Izychev-Venjakob \cite{IzychevVenjakob:2010}) we have an exact sequence
\[
1 \rightarrow G^{ab} \rightarrow K_1'(J[G], I_G) \xrightarrow{L} \tau(I_G) \rightarrow 1.
\]
We denote restrictions of the maps $\theta_{G,J}$ and $\beta_{G,J}$ to subsets by the same symbol. Consider the commutative diagram
\[
\xymatrix{ 1 \ar[r] & G^{ab} \ar[d]_{\theta_{G,J}} \ar[r] & K_1'(J[G], I_G) \ar[r]^{L} \ar[d]_{\theta_{G,J}} &  \tau(I_G) \ar[d]^{\beta_{G,J}} \ar[r] & 0 \\
 1 \ar[r] & G_i^{ab} \ar[r] & \prod_{i=0}^n 1+I_{G_i^{ab}} \ar[r]_{\mathcal{L}} &  I_{G_i^{ab}} \ar[r] & 0}
\]
with exact rows, where $\mathcal{L} = (\mathcal{L}_i)$ is the map
\[
\mathcal{L}_i((x_j)_j) = log\left( \frac{x_i}{ver_i(x_{i-1})}\right).
\]
It is proven in Kakde \cite{Kakde:2011}, corollary 4.4 that this $\mathcal{L}$ makes the diagram commute. One can now proceed as in \emph{loc. cit.} to prove that $(x_i) \in \prod_{i=0}^n 1+ I_{G_i^{ab}}$ lies in $\theta_{G,J}(K_1'(J[G], I_G)$ if and only if it satisfies M1-M3. Next we use the following commutative diagram with exact rows
\[
\xymatrix{ 
1 \ar[r] & K_1'(J[G], I_G) \ar[r] \ar[d]_{\theta_{G,J}} & K_1(J[G]) \ar[r] \ar[d]_{\theta_{G,J}} & K_1(J) \ar[r] \ar[d] & 1 \\
1 \ar[r] & \prod_{i=0}^n 1+I_{G_i^{ab}} \ar[r] & \prod_{i=0}^n K_1(J[G_i^{ab}]) \ar[r] & \prod_{i=0}^n K_1(J) \ar[r] & 1}.
\]
Here the last vertical arrow is $[G:G_i]=p^i$-power map in $i$th factor. Let $(x_i) \in \prod_{i=0}^n K_1(J[G_i^{ab}])$ be a tuple satisfying M1-M3. Let $y_i$ be the image of $x_i$ in $K_1(J)$ under augmentation map i.e. the third arrow in the bottom row. Then M1 implies that $y_i = y_0^{p^i}$. It is also clear that $(x_iy_i^{-1}) \in \prod_{i=0}^n 1+I_{G_i}^{ab}$ and satisfies M1-M3. Theorem \ref{theoremk1} is now clear.  
\end{proof}

\bibliographystyle{plain}
\bibliography{mybib2}

\end{document}